\documentclass[12pt,twoside,reqno,psamsfonts]{amsart}

\usepackage[OT1]{fontenc}
\usepackage{type1cm}
\usepackage{amssymb}
\usepackage[dvips]{graphicx}
\usepackage{geometry}        
\usepackage{version}         

\excludeversion{version:noninteresting_example}

\numberwithin{equation}{section}

\theoremstyle{plain}
\newtheorem{theorem}{Theorem}[section]

\newtheorem{lemma}[theorem]{Lemma}

\theoremstyle{remark}
\newtheorem{remark}[theorem]{Remark}

\newtheorem{example}[theorem]{Example}
\newtheorem*{ack}{Acknowledgement}

\theoremstyle{definition}

\newtheorem{question}[theorem]{Question}

\renewcommand{\atop}[2]{\genfrac{}{}{0pt}{}{#1}{#2}}

\newcommand{\LL}{\mathcal{L}}
\newcommand{\QQ}{\mathcal{Q}}
\newcommand{\HH}{\mathcal{H}}
\newcommand{\PP}{\mathcal{P}}
\newcommand{\MM}{\mathcal{M}}
\newcommand{\EE}{\mathcal{E}}
\newcommand{\CC}{\mathcal{C}}
\newcommand{\R}{\mathbb{R}}

\newcommand{\N}{\mathbb{N}}
\newcommand{\hhh}{\mathtt{h}}
\newcommand{\iii}{\mathtt{i}}
\newcommand{\jjj}{\mathtt{j}}

\newcommand{\eps}{\varepsilon}
\newcommand{\fii}{\varphi}

\newcommand{\ualpha}{\overline{\alpha}}
\newcommand{\lalpha}{\underline{\alpha}}
\newcommand{\ulambda}{\overline{\lambda}}
\newcommand{\llambda}{\underline{\lambda}}

\DeclareMathOperator{\dimm}{dim_M}
\DeclareMathOperator{\dimh}{dim_H}

\DeclareMathOperator{\diml}{dim_L}
\DeclareMathOperator{\dist}{dist}
\DeclareMathOperator{\diam}{diam}
\DeclareMathOperator{\proj}{proj}
\DeclareMathOperator{\conv}{conv}

\DeclareMathOperator{\spt}{spt}

\begin{document}

\title{Dimension and measures on sub-self-affine sets}

\author{Antti K\"aenm\"aki \and Markku Vilppolainen}

\address{Department of Mathematics and Statistics \\
         P.O. Box 35 (MaD) \\
         FI-40014 University of Jyv\"askyl\"a \\
         Finland}
\email{antti.kaenmaki@jyu.fi}
\email{markku.vilppolainen@jyu.fi}

\thanks{AK acknowledges the support of the Academy of Finland (project \#114821)}
\date{\today}

\begin{abstract}
  We show that in a typical sub-self-affine set, the Hausdorff and the Minkowski dimensions coincide and equal the zero of an appropriate topological pressure. This gives a partial positive answer to the question of Falconer. We also study the properties of the topological pressure and the existence and the uniqueness of natural measures supported on a sub-self-affine set.
\end{abstract}

\maketitle

\section{Introduction}

An \emph{iterated function system (IFS)} on $\R^d$ is a finite collection of strictly contractive mappings $f_1,\ldots,f_\kappa \colon \R^d \to \R^d$. For such a system there exists a unique nonempty compact set $E \subset \R^d$ satisfying
\begin{equation} \label{eq:intro_invariant}
  E = \bigcup_{i=1}^\kappa f_i(E),
\end{equation}
see Hutchinson \cite{Hutchinson1981}. When the mappings are similitudes, the set $E$ satisfying \eqref{eq:intro_invariant} is called \emph{self-similar}, and if they are affine, then $E$ is called \emph{self-affine}. There are many works focusing on calculating the dimension and measures of these sets, see, for example \cite{Hutchinson1981, KaenmakiVilppolainen2008, Falconer1988, Falconer1992, HueterLalley1995, FalconerMiao2007, Baranski2007, KaenmakiShmerkin2009}.

Falconer \cite{Falconer1995} introduced a generalization of self-similar sets by relaxing the equality in \eqref{eq:intro_invariant} to inclusion. He termed a compact set satisfying such an inclusion \emph{sub-self-similar}. The same generalization can, of course, be done with self-affine sets. If the mappings $f_1,\ldots,f_\kappa$ are affine, then any nonempty compact set $E \subset \R^d$ satisfying
\begin{equation*}
  E \subset \bigcup_{i=1}^\kappa f_i(E)
\end{equation*}
is called \emph{sub-self-affine}. These sets include many interesting examples, such as sub-self-similar sets, graph directed self-affine sets, unions of self-affine sets, and topological boundaries of self-affine sets. The reader is referred to \cite[\S 2]{Falconer1995} for a more comprehensive list of examples.

For sub-self-similar sets satisfying a sufficient separation condition, the open set condition, Falconer \cite{Falconer1995} proved that the Hausdorff and the Minkowski dimensions coincide and equal the zero of an appropriate topological pressure. In Theorem \ref{thm:ssa_ae_dim}, we show that the same result is true for a typical sub-self-affine set and moreover, such a set carries an invariant measure of full Hausdorff dimension. This gives a partial positive answer to the question of Falconer \cite{Falconer1995}. The proof of Theorem \ref{thm:ssa_ae_dim} is based on the existence of an equilibrium measure. The existence of such measures on a self-affine set was proved by K\"aenm\"aki \cite{Kaenmaki2004}. To our knowledge, it is the first proof for the existence of an ergodic equilibrium measure that is not based on the existence of the Gibbs-type measure. Recall also the question of Falconer \cite[\S 6]{Falconer1994}. Cao, Feng, and Huang \cite{CaoFengHuang2008} have later studied the variational principle in a more general setting. The uniqueness of the equilibrium measure was implicitly asked in \cite{Kaenmaki2004}. In Example \ref{ex:not_unique}, we answer this question in the negative. Sufficient conditions for the uniqueness can be found in Lemma \ref{thm:equal_maps} and Theorem \ref{thm:semiconformal_unique}.

We also study the behavior of the topological pressure. There are at most countably many points where the pressure is not differentiable. We exhibit a nondifferentiable pressure in Example \ref{ex:nondifferentiable}. A sufficient condition for the existence of the derivative is given in Theorem \ref{thm:derivative_exists_when_semiconf}.

\section{Setting and preliminaries}

Throughout the article, we use the following notation: Let $0 < \ualpha < 1$ and $I$ be a finite set with cardinality $\kappa := \# I \ge 2$. Put $I^* = \bigcup_{n=1}^\infty I^n$ and $I^\infty = I^\N$. For each $\iii \in I^*$, there is $n \in \N$ such that $\iii = (i_1,\ldots,i_n) \in I^n$. We call this $n$ the \emph{length} of $\iii$ and we set $|\iii|=n$. The length of elements in $I^\infty$ is infinite. Moreover, if $\iii \in I^*$ and $\jjj \in I^* \cup I^\infty$, then by the notation $\iii\jjj$ we mean the element obtained by juxtaposing the terms of $\iii$ and $\jjj$. For $\iii \in I^*$ and $K \subset I^\infty$, we define $[\iii;K] = \{ \iii\jjj : \jjj \in K \}$ and we call the set $[\iii] := [\iii;I^\infty]$ a \emph{cylinder set} of level $|\iii|$. If $\jjj \in I^* \cup I^\infty$ and $1 \le n < |\jjj|$, we define $\jjj|_n$ to be the unique element $\iii \in I^n$ for which $\jjj \in [\iii]$. If $\jjj \in I^*$ and $n \ge |\jjj|$ then $\jjj|_n = \jjj$. We also set $\iii^- = \iii|_{|\iii|-1}$. We say that the elements $\iii,\jjj \in I^*$ are \emph{incomparable} if $[\iii] \cap [\jjj] = \emptyset$.

Defining
\begin{equation*}
  |\iii - \jjj| =
  \begin{cases}
    \ualpha^{\min\{ k-1 : \iii|_k \ne \jjj|_k \}}, \quad &\iii \ne \jjj \\
    0, &\iii = \jjj
  \end{cases}
\end{equation*}
whenever $\iii,\jjj \in I^\infty$, the couple $(I^\infty,|\cdot|)$ is a compact metric space. We call $(I^\infty,|\cdot|)$ a \emph{symbol space} and an element $\iii = (i_1,i_2,\ldots) \in I^\infty$ a \emph{symbol}. If there is no danger of misunderstanding, let us also call an element $\iii \in I^*$ a symbol. Define the \emph{left shift} $\sigma$ by setting
\begin{equation*}
  \sigma(i_1,i_2,\ldots) = (i_2,i_3,\ldots).
\end{equation*}
It is easy to see that $\sigma$ is a continous transformation on the symbol space. By the notation $\sigma(i_1,\ldots,i_n)$, we mean the symbol $(i_2,\ldots,i_n) \in I^{n-1}$. Observe that to be precise in our definitions, we need an ``empty symbol'', that is, a symbol with zero length.

The singular values $\|A\| = \alpha_1(A) \ge \cdots \ge \alpha_d(A) > 0$ of an invertible matrix $A \in \R^{d\times d}$ are the square roots of the eigenvalues of the positive definite matrix $A^*A$, where $A^*$ is the transpose of $A$. Given such a matrix $A$, we define the singular value function to be
\begin{equation*}
  \fii^t(A) = \alpha_1(A) \cdots \alpha_l(A) \alpha_{l+1}(A)^{t-l},
\end{equation*}
where $0 \le t < d$ and $l = \lfloor t \rfloor$ is the integer part of $t$. When $t \ge d$, we set $\fii^t(A) = |\det(A)|^{t/d}$ for completeness.

For each $i \in I$, fix an invertible matrix $A_i \in \R^{d\times d}$ such that $\|A_i\| \le \ualpha$. Clearly the products $A_\iii = A_{i_1}\cdots A_{i_n}$ are also invertible for all $\iii \in I^n$ and $n \in \N$. If we let
$\lalpha = \min_{i \in I} \alpha_d(A_i) > 0$, it follows that
\begin{equation} \label{eq:cylinder1}
  \fii^t(A_\iii)\lalpha^{\delta |\iii|}
  \le \fii^{t+\delta}(A_\iii)
  \le \fii^t(A_\iii)\ualpha^{\delta |\iii|}
\end{equation}
for all $t,\delta \ge 0$ and $\iii \in I^*$. According to \cite[Corollary V.1.1]{Temam1988} and \cite[Lemma 2.1]{Falconer1988}, we have
\begin{equation} \label{eq:cylinder2}
  \fii^t(A_{\iii\jjj}) \le \fii^t(A_\iii) \fii^t(A_\jjj)
\end{equation}
for all $t \ge 0$ and $\iii,\jjj \in I^*$.

For $K \subset I^\infty$ we set $K_n = \{ \iii|_n \in I^n : \iii \in K \}$. If $\sigma(K) \subset K$, then it follows that
\begin{equation} \label{eq:sub_shift}
  \iii\jjj \in K_{n+m} \implies \iii \in K_n \text{ and } \jjj \in K_m
\end{equation}
for all $\iii \in I^n$ and $\jjj \in I^m$. The converse does not necessarily hold: choose $I = \{ 0,1 \}$ and define $K = \{ (0,0,\ldots), (1,1,\ldots) \} \subset I^\infty$ for a trivial counterexample. We also set $K_* = \bigcup_{n=1}^\infty K_n$. Now \eqref{eq:cylinder2} together with \eqref{eq:sub_shift} implies that
\begin{align*}
  \sum_{\iii \in K_{n+m}} \fii^t(A_\iii) &\le \sum_{\iii \in K_{n+m}} \fii^t(A_{\iii|_n})\fii^t(A_{\sigma^n(\iii)}) \\ &\le \sum_{\atop{\iii \in K_n}{\jjj \in K_m}} \fii^t(A_\iii)\fii^t(A_\jjj) = \sum_{\iii \in K_n} \fii^t(A_\iii) \sum_{\jjj \in K_m} \fii^t(A_\jjj)
\end{align*}
for all $t \ge 0$ and $n,m \in \N$. This observation allows us to give the following definition. Given a set $K \subset I^\infty$ with $\sigma(K) \subset K$ and $t \ge 0$, we define the \emph{topological pressure} to be
\begin{equation*}
  P_K(t) = \lim_{n \to \infty} \tfrac{1}{n} \log\sum_{\iii \in K_n} \fii^t(A_\iii).
\end{equation*}
The limit above exists by the standard theory of subadditive sequences. Furthermore, it holds that $P_K(t) = \inf_{n \in \N} \tfrac{1}{n} \log\sum_{\iii \in K_n} \fii^t(A_\iii)$. To simplify the notation, we let $P(t) = P_{I^\infty}(t)$.

\begin{lemma} \label{thm:P_convex}
  Suppose that for each $i \in I$ there is an invertible matrix $A_i \in \R^{d \times d}$. If $K \subset I^\infty$ is a nonempty set with $\sigma(K) \subset K$, then the function $P_K \colon [0,\infty) \to \R$ is continuous, and convex on the connected components of $[0,\infty) \setminus \{ 1,\ldots,d \}$. Furthermore, if $\| A_i \| < 1$ for all $i \in I$, then $P_K$ is strictly decreasing and there exists a unique $t \ge 0$ for which $P_K(t)=0$.
\end{lemma}

\begin{proof}
  If we let $\lalpha = \min_{i \in I} \alpha_d(A_i) > 0$ and $\ualpha = \max_{i \in I} \alpha_1(A_i)$, then \eqref{eq:cylinder1} implies
  \begin{equation*}
    -\delta \log\ualpha \le P_K(t) - P_K(t + \delta) \le -\delta \log\lalpha
  \end{equation*}
  for all $t,\delta \ge 0$. It follows that $P_K$ is continuous. Furthermore, if the matrices are contractive, then $-\delta \log\ualpha > 0$ and $P_K$ is strictly decreasing with $\lim_{t \to \infty} P_K(t) = -\infty$. Since $P_K(0) = \lim_{n \to \infty} \tfrac1n \log \# K_n \ge 0$, the last claim is immediate.

  To prove the claimed convexity, take $\gamma \in (0,1)$ and $0 \le t_1<t_2<d$. Letting $l_1 = \lfloor t_1 \rfloor$, $l_2 = \lfloor t_2 \rfloor$, and $l = \lfloor \gamma t_1 + (1-\gamma)t_2 \rfloor$, we have for each $\iii \in K_*$
  \begin{equation*}
    \alpha_1(A_\iii) \cdots \alpha_l(A_\iii)
    \alpha_{l+1}(A_\iii)^{t_1-l} = \fii^{t_1}(A_\iii)
    \frac{\alpha_{l_1+1}(A_\iii)^{l_1-t_1}\alpha_{l_1+1}(A_\iii)
          \cdots\alpha_l(A_\iii)}
         {\alpha_{l+1}(A_\iii)^{l-t_1}}
  \end{equation*}
  and
  \begin{equation*}
    \alpha_1(A_\iii) \cdots \alpha_l(A_\iii)
    \alpha_{l+1}(A_\iii)^{t_2-l} = \fii^{t_2}(A_\iii)
    \frac{\alpha_{l+1}(A_\iii)^{t_2-l}}
         {\alpha_{l+1}(A_\iii)\cdots\alpha_{l_2}(A_\iii)
          \alpha_{l_2+1}(A_\iii)^{t_2-l_2}}.
  \end{equation*}
  Hence
  \begin{equation} \label{eq:P_convex}
  \begin{split}
    \fii^{\gamma t_1 + (1-\gamma)t_2}(A_\iii)
    &= \bigl( \alpha_1(A_\iii) \cdots
    \alpha_l(A_\iii) \alpha_{l+1}(A_\iii)^{t_1-l} \bigr)^\gamma \\
    &\;\;\,\,\:\: \bigl( \alpha_1(A_\iii) \cdots
    \alpha_l(A_\iii) \alpha_{l+1}(A_\iii)^{t_2-l} \bigr)^{1-\gamma} \\
    &\le \biggl( \fii^{t_1}(A_\iii)
    \frac{\alpha_{l_1+1}(A_\iii)^{l-t_1}}{\alpha_{l+1}(A_\iii)^{l-t_1}}
    \biggr)^\gamma
    \biggl( \fii^{t_2}(A_\iii)
    \frac{\alpha_{l+1}(A_\iii)^{t_2-l}}{\alpha_{l_2+1}(A_\iii)^{t_2-l}}
    \biggr)^{1-\gamma}.
  \end{split}
  \end{equation}
  The convexity on the connected components of $[0,d] \setminus \{1,\ldots,d\}$ is now an immediate consequence of H\"older's inequality. Indeed, choosing $0 \le t_1<t_2<d$ such that $\lfloor t_1 \rfloor = \lfloor t_2 \rfloor$, we obtain from \eqref{eq:P_convex} that for each $n \in \N$
  \begin{equation*}
    \sum_{\iii \in K_n} \fii^{\gamma t_1 + (1-\gamma)t_2}(A_\iii) \le
    \biggl( \sum_{\iii \in K_n} \fii^{t_1}(A_\iii) \biggr)^\gamma
    \biggl( \sum_{\iii \in K_n} \fii^{t_2}(A_\iii) \biggr)^{1-\gamma}.
  \end{equation*}
  The claim follows since the infimum of a family of convex functions is convex. The convexity on $[d,\infty)$ follows by a similar reasoning.
\end{proof}

\begin{remark}
  Observe that \eqref{eq:P_convex} implies
  \begin{equation*}
    P_K\bigl( \gamma t_1 + (1-\gamma) t_2 \bigr) \le \gamma
    P_K(t_1) + (1-\gamma)P_K(t_2) + d\log(\ualpha/\lalpha)
  \end{equation*}
  for all $0 \le t_1 < t_2 < d$ and $\gamma \in (0,1)$. It follows from \cite[\S 2]{Cholewa1984} that there exists a convex function $\tilde{P}_K \colon [0,\infty) \to \R$ such that $|P_K(t) - \tilde{P}_K(t)| \le \tfrac{d}{2} \log(\ualpha/\lalpha)$ for all $0 \le t < d$.
\end{remark}

Suppose that for each $i \in I$ there is an invertible matrix $A_i \in \R^{d \times d}$ with $\|A_i\| \le \ualpha < 1$. For $a=(a_1,\ldots,a_\kappa) \in \R^{d\kappa}$, where $a_i \in \R^d$ is a translation vector and $\kappa = \# I$, we define a \emph{projection mapping} $\pi_a \colon I^\infty \to \R^d$ by setting
\begin{equation*}
  \pi_a(\iii) = \sum_{n=1}^\infty A_{\iii|_{n-1}}a_{i_n}
\end{equation*}
as $\iii = (i_1,i_2,\ldots)$. The mapping $\pi_a$ is clearly continuous. The collection of contractive affine mappings $\{ A_i + a_i \}_{i \in I}$ is called an \emph{affine iterated function system (affine IFS)}.

If $K \subset I^\infty$ is a nonempty compact set with $\sigma(K) \subset K$, then we define $E_K = \pi_a(K)$ and call this set \emph{sub-self-affine}. We also set $E = E_{I^\infty}$ and call this set \emph{self-affine}. The compact set $E_K$ satisfies
\begin{equation} \label{eq:subinvariant_set}
  E_K \subset \bigcup_{i \in I} (A_i+a_i)(E_K).
\end{equation}
This is an immediate consequence of the fact that
\begin{equation*}
  \pi_a(i\iii) = (A_i + a_i)\sum_{n=1}^\infty A_{\iii|_{n-1}}a_{i_n}
  = (A_i + a_i)\pi_a(\iii)
\end{equation*}
whenever $\iii \in K$ and $i \in K_1$. The converse is also true. Namely, if $E'$ is a compact set satisfying \eqref{eq:subinvariant_set}, then for the compact set $K = \bigcap_{n=0}^\infty \{ \iii \in I^\infty : \sigma^n(\iii) \in \pi_a^{-1}(E') \}$ we clearly have $\sigma(K) \subset K$ and $\pi_a(K) \subset E'$. To see that $E' \subset \pi_a(K)$, pick $x_0 \in E'$ and use \eqref{eq:subinvariant_set} repeatedly to discover a symbol $\iii = (i_1,i_2,\ldots) \in I^\infty$ such that for each $n \in \N$ there exists $x_n \in E'$ so that
\begin{equation*}
  x_0 = (A_{i_1}+a_{i_1}) \cdots (A_{i_n}+a_{i_n})(x_n) = A_{\iii|_n} x_n + A_{\iii|_{n-1}}a_{i_n} + \cdots + a_{i_1}.
\end{equation*}
Since $|A_{\iii|_n} x_n| \to 0$ as $n \to \infty$, we have $x_0=\pi_a(\iii)$. That $\iii \in K$ follows from the fact that $\pi_a\bigl( \sigma^n(\iii) \bigr) = x_n$ for all $n \in \N$. See also \cite[Proposition 2.1]{Falconer1995} and \cite[\S 3]{Bandt1989}.

Since the self-affine set $E$ satisfies \eqref{eq:subinvariant_set} with equality, it is also called an \emph{invariant set} or \emph{attractor} of the affine IFS $\{ A_i+a_i \}_{i \in I}$. It follows from \cite[\S 3.1]{Hutchinson1981} that $E$ is the only nonempty compact set satisfying such an equality. If there is no danger of misunderstanding, the image of a cylinder set $E_\iii = \pi_a([\iii])$ will also be called a cylinder set.

Recalling Lemma \ref{thm:P_convex}, we define the \emph{singularity dimension} to be the unique $t \ge 0$ for which $P_K(t)=0$. See also \cite[Proposition 4.1]{Falconer1988}, \cite[Proposition 3.2]{Falconer1995}, and \cite[\S 2]{KaenmakiShmerkin2009}. Inspecting the proof of \cite[Theorem 5.4]{Falconer1988}, we see that the singularity dimension serves as an upper bound for the upper Minkowski dimension of $E_K$.

\section{Equilibrium measures}

We denote the collection of all Borel probability measures on $I^\infty$ by $\MM(I^\infty)$. We set $\MM_\sigma(I^\infty) = \{ \mu \in \MM(I^\infty) : \mu \text{ is $\sigma$-invariant} \}$, where the \emph{$\sigma$-invariance of $\mu$} means that $\mu([\iii]) = \mu\bigl( \sigma^{-1}([\iii]) \bigr) = \sum_{i \in I} \mu([i\iii])$ for all $\iii \in I^*$. Observe that if $\mu \in \MM_\sigma(I^\infty)$, then $\mu(A)=\mu\bigl( \sigma^{-1}(A) \bigr)$ for all Borel sets $A \subset I^\infty$ by \cite[Theorem 5.4]{Bauer2001}. Furthermore, we set $\EE_\sigma(I^\infty) = \{ \mu \in \MM_\sigma(I^\infty) : \mu \text{ is ergodic} \}$, where the \emph{ergodicity of $\mu$} means that $\mu(A) = 0$ or $\mu(A) = 1$ for every Borel set $A \subset I^\infty$ with $A = \sigma^{-1}(A)$. Recall from \cite[Theorem 6.10]{Walters1982} that the set $\MM_\sigma(I^\infty)$ is compact and convex with ergodic measures as its extreme points.

For $K \subset I^\infty$ we also set $\MM_\sigma(K) = \{ \mu \in \MM_\sigma(I^\infty) : \spt(\mu) \subset K \}$ and $\EE_\sigma(K) = \{ \mu \in \EE_\sigma(I^\infty) : \spt(\mu) \subset K \}$. Here $\spt(\mu)$ is the support of $\mu$, that is, the smallest closed set $F \subset I^\infty$ for which $\mu(I^\infty \setminus F)=0$. Observe that if $K \subset I^\infty$ is a nonempty compact set, then $\MM_\sigma(K)$ is nonempty by \cite[Corollary 6.9.1]{Walters1982}. It is also compact and convex with $\EE_\sigma(K)$ as the set of its extreme points.

\begin{remark} \label{rem:erg_in_K}
  (1) Let $K \subset I^\infty$ be a nonempty compact set with $\sigma(K) \subset K$. If $\mu \in \EE_\sigma(I^\infty)$ satisfies $\mu(I^\infty \setminus K) > 0$, then the invariance of $\mu$ yields $\mu(K_0) = \mu(I^\infty \setminus K)$, where $K_0 = \bigcap_{n=1}^\infty \sigma^{-n}(I^\infty \setminus K) \subset I^\infty \setminus K$. Furthermore, since $\sigma^{-1}(K_0) = K_0$, the ergodicity of $\mu$ yields $\mu(I^\infty \setminus K) = 1$.

  (2) There exists a measure $\mu \in \EE_\sigma(I^\infty)$ and a nonempty compact set $K \subset I^\infty$ with $\sigma(K) \subset K$ so that $\mu(K)=0$. Namely, let $I = \{ 0,1 \}$ and define $\mu \in \MM(I^\infty)$ by setting $\mu(\{ (0,1,0,1,\ldots) \}) = \tfrac12 = \mu(\{ (1,0,1,0,\ldots) \})$. Since $\mu([\iii]) = \mu([0\iii]) + \mu([1\iii])$ for all $\iii \in I^*$, the measure $\mu$ is invariant. Furthermore, let $A \subset I^\infty$ be a Borel set with $A=\sigma^{-1}(A)$ and $0<\mu(A)<1$. It follows that $\mu(A)=\tfrac12$. If $(0,1,0,1,\ldots) \in A = \sigma^{-1}(A)$, then also $(1,0,1,0,\ldots) \in A$, which clearly cannot be the case. Similarly the other way around. Hence $\mu$ is ergodic. Choosing now $K=\{ (0,0,\ldots),(1,1,\ldots) \}$ (or $K=\{ \iii \in I^\infty : \sigma^{n-1}(\iii|_{n+1}) \ne (0,1) \text{ for all } n \in \N \}$), we have found the desired compact set.
\end{remark}

We define a concave function $H \colon [0,1] \to \R$ by setting $H(x) = -x\log x$ when $0 < x \le 1$ and $H(0)=0$. Notice that $0 \le H(x) \le \tfrac{1}{e}$ for all $x \in [0,1]$. The following lemma is an immediate consequence of \cite[Lemmas 2.3 and 2.2]{Kaenmaki2004}.

\begin{lemma} \label{thm:gen_sub_add}
  If $\mu \in \MM(I^\infty)$, then
  \begin{equation*}
    \tfrac1n \sum_{\iii \in I^n} H\bigl( \mu([\iii]) \bigr) \le \tfrac{1}{kn} \sum_{j=0}^{n-1} \sum_{\iii \in I^k} H\bigl( \mu \circ \sigma^{-j}([\iii]) \bigr) + \tfrac{3k}{n}\log\# I
  \end{equation*}
  for all integers $0<k<n$. In addition, if $t \ge 0$ and for each $i \in I$ there is an invertible matrix $A_i \in \R^{d \times d}$, then
  \begin{equation*}
    \tfrac1n \sum_{\iii \in I^n} \mu([\iii])\log\fii^t(A_\iii) \le \tfrac{1}{kn} \sum_{j=0}^{n-1} \sum_{\iii \in I^k} \mu \circ \sigma^{-j}([\iii])\log\fii^t(A_\iii) + \tfrac{3k}{n}\log\lalpha^{-t}
  \end{equation*}
  for all integers $0<k<n$. Here $\lalpha = \min_{i \in I} \alpha_d(A_i)>0$.
\end{lemma}

If $\mu \in \MM_\sigma(I^\infty)$, then we define the \emph{entropy} of $\mu$ by setting
\begin{equation*}
  h(\mu) = \lim_{n \to \infty} \tfrac{1}{n} \sum_{\iii \in I^n} H\bigl( \mu([\iii]) \bigr).
\end{equation*}
In addition, if for each $i \in I$ there is an invertible matrix $A_i \in \R^{d \times d}$, then for every $t \ge 0$ we define the \emph{$t$-energy} of $\mu$ by setting
\begin{equation*}
  \Lambda^t(\mu) = \lim_{n \to \infty} \tfrac{1}{n}\sum_{\iii \in I^n} \mu([\iii]) \log \fii^t(A_\iii).
\end{equation*}
The above limits exist since
\begin{align*}
  \limsup_{n \to \infty} \tfrac{1}{n} \sum_{\iii \in I^n} H\bigl( \mu([\iii]) \bigr) &\le \inf_{k \in \N} \tfrac{1}{k} \sum_{\iii \in I^k} H\bigl( \mu([\iii]) \bigr), \\
  \limsup_{n \to \infty} \tfrac{1}{n}\sum_{\iii \in I^n} \mu([\iii]) \log \fii^t(A_\iii) &\le \inf_{k \in \N} \tfrac{1}{k}\sum_{\iii \in I^k} \mu([\iii]) \log \fii^t(A_\iii)
\end{align*}
by the invariance of $\mu$ and Lemma \ref{thm:gen_sub_add}.

Suppose $K \subset I^\infty$ is a nonempty compact set and $\mu \in \MM(K)$. Since, trivially, $\mu([\iii])=0$ for all $I^* \setminus K_*$, we may write Lemma \ref{thm:gen_sub_add} and the definitions of entropy and $t$-energy by using $K_n$ instead of $I^n$. For our purposes the assumption $\mu \in \MM(K)$ is natural, see Remark \ref{rem:erg_in_K}. Observe that the mapping $\mu \mapsto \Lambda^t(\mu)$ defined on $\MM_\sigma(K)$ is the infimum of continuous and affine mappings. Since also the mapping $\mu \mapsto h(\mu)$ defined on $\MM_\sigma(K)$ is upper semicontinuous and affine (see \cite[Theorems 8.1 and 8.2]{Walters1982} or \cite[proof of Theorem 4.1]{Kaenmaki2004}), the mapping $\mu \mapsto h(\mu) + \Lambda^t(\mu)$ defined on $\MM_\sigma(K)$ is upper semicontinuous and affine for each $t \ge 0$. Finally, it is easy to see that $0 \le h_K(\mu) \le \log\# I$ and $t\log\lalpha \le \Lambda^t(\mu) \le t\log\ualpha$ for all $\mu \in \MM_\sigma(K)$ and $t \ge 0$, where $\lalpha = \min_{i \in I} \alpha_d(A_i) > 0$ and $\ualpha = \max_{i \in I} \alpha_1(A_i)$.

If $\mu \in \MM(I^\infty)$ and $t \ge 0$, then for each $n \in \N$ and $C_n \subset I^n$ Jensen's inequality implies
\begin{equation} \label{eq:jensen_appl}
\begin{split}
  \sum_{\iii \in C_n} &\mu([\iii])\biggl( -\log\mu([\iii]) + \log\fii^t(A_\iii) - \log\sum_{\jjj \in C_n} \fii^t(A_\jjj) \biggr) \\
  &= \sum_{\iii \in C_n} \beta(\iii) H\biggl( \frac{\mu([\iii])}{\beta(\iii)} \biggr) \le H\biggl( \sum_{\iii \in C_n} \beta(\iii) \frac{\mu([\iii])}{\beta(\iii)} \biggr) \in [0,\tfrac{1}{e}],
\end{split}
\end{equation}
where $\beta(\iii) = \fii^t(A_\iii) / \sum_{\jjj \in C_n} \fii^t(A_\jjj)$ for all $\iii \in C_n$. In particular, if $K \subset I^\infty$ is a nonempty compact set with $\sigma(K) \subset K$, then \eqref{eq:jensen_appl} applied with $C_n = K_n$ yields
\begin{equation*}
  P_K(t) \ge h(\mu) + \Lambda^t(\mu)
\end{equation*}
for all $\mu \in \MM_\sigma(K)$ and $t \ge 0$. A measure $\mu \in \MM_\sigma(K)$ is called a \emph{$(K,t)$-equilibrium measure} if
\begin{equation*}
  P_K(t) = h(\mu) + \Lambda^t(\mu).
\end{equation*}
To simplify the notation, we speak about $t$-equilibrium measures when $K=I^\infty$. From now on, without mentioning it explicitly, we assume that for each $i \in I$ there is an invertible matrix $A_i \in \R^{d \times d}$.

\begin{theorem}
  If $K \subset I^\infty$ is a nonempty compact set with $\sigma(K) \subset K$ and $t \ge 0$, then there exists an ergodic $(K,t)$-equilibrium measure.
\end{theorem}

\begin{proof}
  To prove the theorem, we use an approach similar to that in \cite[Theorem 2.6]{Kaenmaki2004}. For each $n \in \N$ we define a Borel probability measure
  \begin{equation*}
    \nu_n = \frac{\sum_{\iii \in K_n} \fii^t(A_\iii) \delta_{\iii\hhh_\iii}}{\sum_{\iii \in K_n} \fii^t(A_\iii)},
  \end{equation*}
  where $\delta_{\iii\hhh_\iii}$ is the Dirac measure and for each $\iii \in K_n$ the symbol $\hhh_\iii \in I^\infty$ is chosen such that $\iii\hhh_\iii \in K$. Now with the measure $\nu_n \in \MM(K)$, we have an equality in \eqref{eq:jensen_appl} (when $C_n$ is chosen to be $K_n$). Furthermore, we set
  \begin{equation*}
    \mu_n = \tfrac{1}{n} \sum_{j=0}^{n-1} \nu_n \circ \sigma^{-j}
  \end{equation*}
  for all $n \in \N$. By going into a subsequence, if necessary, it follows that $\{ \mu_n \}_n$ converges weakly to some $\mu \in \MM_\sigma(K)$. According to Lemma \ref{thm:gen_sub_add} and the concavity of $H$, we have
  \begin{equation} \label{eq:eqm1}
    \tfrac{1}{n}\sum_{\iii \in K_n} H\bigl( \nu_n([\iii]) \bigr) \le \tfrac{1}{k}\sum_{\iii \in K_k} H\bigl( \mu_n([\iii]) \bigr) + \tfrac{3k}{n}\log\# I
  \end{equation}
  for all integers $0<k<n$. Letting $\lalpha = \min_{i \in I}\alpha_d(A_i)>0$, Lemma \ref{thm:gen_sub_add} also implies
  \begin{equation} \label{eq:eqm2}
    \tfrac{1}{n} \sum_{\iii \in K_n} \nu_n([\iii])\log\fii^t(A_\iii) \le \tfrac{1}{k} \sum_{\iii \in K_k} \mu_n([\iii])\log\fii^t(A_\iii) + \tfrac{3k}{n}\log\lalpha^{-t}
  \end{equation}
  for all integers $0<k<n$. Now \eqref{eq:jensen_appl} (with the equality), \eqref{eq:eqm1}, and \eqref{eq:eqm2} yield
  \begin{equation*}
    \tfrac{1}{n}\log\sum_{\iii \in K_n} \fii^t(A_\iii) \le \tfrac{1}{k}\sum_{\iii \in K_k} H\bigl( \mu_n([\iii]) \bigr) + \tfrac{1}{k} \sum_{\iii \in K_k} \mu_n([\iii])\log\fii^t(A_\iii) + \tfrac{3k}{n}\log\# I\lalpha^{-t}
  \end{equation*}
  for all integers $0<k<n$. Letting $n \to \infty$ along the chosen subsequence, we get
  \begin{equation*}
    P_K(t) \le \tfrac{1}{k}\sum_{\iii \in K_k} H\bigl( \mu([\iii]) \bigr) + \tfrac{1}{k} \sum_{\iii \in K_k} \mu([\iii])\log\fii^t(A_\iii).
  \end{equation*}
  Recall also \cite[Theorem 1.24]{Mattila1995} and the fact that cylinders are both open and closed. Letting $k \to \infty$, we have shown that $\mu$ is a $(K,t)$-equilibrium measure.

  Since the mapping $\mu \mapsto h(\mu) + \Lambda^t(\mu)$ defined on $\MM_\sigma(K)$ is upper semicontinuous and affine, the set of all $(K,t)$-equilibrium measures is compact and convex. Moreover, for a $(K,t)$-equilibrium measure $\mu$, by Choquet's Theorem (\cite[\S 3]{Phelps1966}), there exists a Borel probability measure $\tau_\mu$ on $\EE_\sigma(K)$ such that
  \begin{equation*} 
    h(\mu) + \Lambda^t(\mu) = \int_{\EE_\sigma(K)} \bigl( h(\eta) + \Lambda^t(\eta) \bigr) d\tau_\mu(\eta).
  \end{equation*}
  This implies the existence of an ergodic $(K,t)$-equilibrium measure. See also \cite[Theorem 4.1]{Kaenmaki2004}.
\end{proof}

If $K \subset I^\infty$ is a nonempty compact set with $\sigma(K) \subset K$, then $\mu \in \MM_\sigma(K)$ is called a \emph{$(K,t)$-semiconformal measure} if it satisfies the following Gibbs-type property: there exists a constant $c \ge 1$ such that
\begin{equation} \label{eq:semiconformal}
  c^{-1}e^{-|\iii|P_K(t)}\fii^t(A_\iii) \le \mu([\iii]) \le ce^{-|\iii|P_K(t)}\fii^t(A_\iii)
\end{equation}
for all $\iii \in K_*$. To simplify the notation, we speak about $t$-semiconformal measures when $K=I^\infty$. For the motivation of the term ``semiconformal'', the reader is referred to \cite[Lemma 1]{Falconer1992}. Since a $(K,t)$-semiconformal measure $\mu$ satisfies
\begin{equation*}
  h(\mu) + \Lambda^t(\mu) \ge \lim_{n \to \infty} \tfrac{1}{n}\sum_{\iii \in K_n} \mu([\iii])\bigl( -\log\mu([\iii]) + \log c^{-1}e^{nP_K(t)} \mu([\iii]) \bigr) = P_K(t),
\end{equation*}
a $(K,t)$-semiconformal measure is always a $(K,t)$-equilibrium measure.

We will discover that semiconformal measures may or may not exist: If for given $t \ge 0$ there exists a constant $D \ge 1$ such that
\begin{equation*}
  D^{-1}\fii^t(A_\iii)\fii^t(A_\jjj) \le \fii^t(A_{\iii\jjj})
\end{equation*}
for all $\iii, \jjj \in I^*$, then there exists an ergodic $t$-semiconformal measure. Take account of Remark \ref{rem:positive_condition}. Moreover, if a $t$-semiconformal measure is ergodic, then it is the only $t$-semiconformal measure. These facts follow from \cite[Theorem 2.2]{KaenmakiVilppolainen2008} by a minor modification. More precisely, in \cite{KaenmakiVilppolainen2008} it was assumed that the parameter $t$ is an exponent, but an examination of the proof reveals that this fact is not required. In Example \ref{ex:no_semiconformal}, it is shown that semiconformal measures do not always exist. Also, even if there exists a semiconformal measure, it is not necessarily ergodic, see Example \ref{ex:not_unique}.

\begin{remark} \label{rem:finite_measure}
 Suppose $K \subset I^\infty$ is a nonempty compact set with $\sigma(K) \subset K$. If a $(K,t)$-semiconformal measure exists when $P_K(t)=0$, then it is easy to see that $\HH^t(E_K) < \infty$. Here $\HH^t$ denotes the $t$-dimensional Hausdorff measure, see \cite[\S 4.3]{Mattila1995}. Indeed, it follows from \cite[proof of Proposition 5.1]{Falconer1988} that there exists a constant $c$, depending only on the dimension of the ambient space, such that
  \begin{equation*}
    \HH^t(E_K) \le c \liminf_{n \to \infty} \sum_{\iii \in K_n} \fii^t(A_\iii).
  \end{equation*}
  From this, the claim follows immediately. Moreover, if $P_K(t)=0$ for $0 \le t \le 1$, then the existence of a $(K,t)$-semiconformal measure implies $\PP^t(E_K) < \infty$. Here $\PP^t$ denotes the $t$-dimensional packing measure, see \cite[\S 5.10]{Mattila1995}. This fact follows from \cite[Proposition 2.2(d)]{Falconer1997}, since for every $x \in E_K$ and $0<r<\diam(E_K)$ 
  \begin{equation*}
    \frac{\mu \circ \pi_a^{-1}\bigl( B(x,r) \bigr)}{r^t} \ge
    \frac{\mu([\iii|_n])}{\diam(E_{\iii|_n})^t} \ge
    c\frac{\fii^t(A_{\iii|_n})}{\alpha_1(A_{\iii|_n})^t} = c > 0,
  \end{equation*}
  where $x = \pi_a(\iii)$ for some $\iii \in K$, $\mu$ is a $t$-semiconformal measure, and $n$ is the smallest integer for which $E_{\iii|_n} \subset B(x,r)$. It seems difficult to extend this to the case $t > 1$.
\end{remark}

The following two results give sufficient conditions for the uniqueness of the equilibrium measure. A self-affine set carrying two different ergodic equilibrium measures is given in Example \ref{ex:not_unique}.

\begin{lemma} \label{thm:equal_maps}
  Suppose that there is an invertible matrix $A \in \R^{d \times d}$ such that $A_i = A$ for all $i \in I$. Then for each $t \ge 0$, the uniform Bernoulli measure is the unique $t$-equilibrium measure.
\end{lemma}

\begin{proof}
  In this case, for a given $n \in \N$, the value of $\fii^t(A_\iii)$ is independent of the specific $\iii\in I^n$ chosen. Using this, one easily finds that
  \begin{equation*}
    P(t) = \log \kappa + \lim_{n \to \infty} \tfrac{1}{n}
    \log\fii^t(A^n),
  \end{equation*}
  and, for any $\mu\in \MM_\sigma(I^\infty)$,
  \begin{equation*}
    \Lambda^t(\mu) = \lim_{n \to \infty} \tfrac{1}{n} \log\fii^t(A^n).
  \end{equation*}
  Applying Jensen's inequality, we have
  \begin{equation*}
    \tfrac{1}{\kappa^n} \sum_{\iii \in I^n} H\bigl( \mu([\iii]) \bigr)
    \le H(\tfrac{1}{\kappa^n}) = \tfrac{n}{\kappa^n} \log\kappa
  \end{equation*}
  with equality if and only if $\mu([\iii]) = \tfrac{1}{\kappa^n}$ for every $\iii \in I^n$ whenever $n \in \N$. Hence $h(\mu) \le \log\kappa$ for every $\mu \in \MM_\sigma(I^\infty)$ and $h(\mu) = \log\kappa$ if and only if $\mu$ is the uniform Bernoulli measure. The proof is finished.
\end{proof}

We remark that generalizing the result of Lemma \ref{thm:equal_maps} for an arbitrary nonempty set $K \subset I^\infty$ satisfying $\sigma(K) \subset K$ seems to be difficult. The following theorem shows that if a $(K,t)$-semiconformal measure is ergodic, then it is the only $(K,t)$-equilibrium measure.

\begin{theorem} \label{thm:semiconformal_unique}
  Suppose that $K \subset I^\infty$ is a nonempty compact set with $\sigma(K) \subset K$, $t \ge 0$, and $c \ge 1$. If $\mu \in \MM_\sigma(K)$ satisfies $\mu([\iii]) \ge c^{-1}e^{-|\iii|P_K(t)}\fii^t(A_\iii)$ for all $\iii \in K_*$, then any $(K,t)$-equilibrium measure is absolutely continuous with respect to $\mu$. Moreover, if $\mu$ lies in the convex hull of a countable family of ergodic $(K,t)$-equilibrium measures, then the closure of the convex hull is precisely the set of all $(K,t)$-equilibrium measures. In particular, if $\mu$ is itself an ergodic $(K,t)$-equilibrium measure, then it is the only $(K,t)$-equilibrium measure.
\end{theorem}

\begin{proof}
  To prove the first claim, we follow some of the ideas of Bowen \cite{Bowen1975}. Assume to the contrary that there exists a Borel set $A \subset I^\infty$ such that $\mu(A)=0$ and $\nu(A)>0$, where $\nu$ is a $(K,t)$-equilibrium measure. Choose $0 < \eps < \exp\bigl( -2(\log c + \tfrac{2}{e})/\nu(A) \bigr)$. Now there exists a sequence of symbols $\{ \iii_k \}$ such that
  \begin{equation*}
    A \subset \bigcup_{k=1}^\infty [\iii_k] =: \tilde{C}
  \end{equation*}
  and $\mu(\tilde{C}) = \mu(\tilde{C} \setminus A) < \eps$. Since $\nu(\tilde{C}) \ge \nu(A) > 0$, there exists $N \in \N$ large enough such that $\nu(\tilde{C}_N) \ge \nu(A)/2$, where
  \begin{equation*}
    \tilde{C}_N = \bigcup_{k=1}^N [\iii_k].
  \end{equation*}
  Since we are now dealing with a finite union, we can rewrite $\tilde{C}_N$ as a union of cylinders of the same length $n$. Let us call the collection of these symbols $C_n$, and also set $C_n' = \tilde{C}_N$. Note that $\nu(C_n') \ge \nu(A)/2$ and $\mu(C_n') \le \mu(\tilde{C}) < \eps$. Let us also set $K_n' = \bigcup_{\iii \in K_n}[\iii]$.

  Using the fact that $\nu$ is a $(K,t)$-equilibrium measure, applying \eqref{eq:jensen_appl}, and recalling the assumption of the measure $\mu$, we have
  \begin{align*}
    nP_K(t) &\le -\sum_{\iii \in K_n} \nu([\iii])\log\nu([\iii]) +
    \sum_{\iii \in K_n} \nu([\iii])\log\fii^t(A_\iii) \\
    &= \sum_{\iii \in K_n \cap C_n} \nu([\iii])\bigl( -\log\nu([\iii]) +
    \log\fii^t(A_\iii) \bigr) \\ &\qquad\quad + \sum_{\iii \in K_n
      \setminus C_n} \nu([\iii])\bigl( -\log\nu([\iii]) +
    \log\fii^t(A_\iii) \bigr) \\
    &\le \nu(K_n' \cap C_n')\log\sum_{\iii \in K_n \cap C_n}\fii^t(A_\iii) +
    \nu(K_n' \setminus C_n')\log\sum_{\iii \in K_n \setminus
      C_n}\fii^t(A_\iii) + \tfrac{2}{e} \\
    &\le \nu(K \cap C_n') \log\mu(C_n') + \nu(K_n' \setminus C_n')
    \log\mu(K_n' \setminus C_n') + nP_K(t) + \log c + \tfrac{2}{e}.
  \end{align*}
  Since $\spt(\nu) \subset K$, this implies
  \begin{equation*}
    0 \le \tfrac12 \nu(A) \log\mu(C_n') + \log c + \tfrac{2}{e} < 0,
  \end{equation*}
  which is a contradiction.

  For the second claim, suppose that $\mu$ lies in the convex hull of $\{ \nu_i \}_{i=1}^\infty$, where each $\nu_i$ is an ergodic $(K,t)$-equilibrium measure. Denote the closure of the convex hull by $\CC$. Since the set of all $(K,t)$-equilibrium measures is compact and convex, all the measures in $\CC$ are $(K,t)$-equilibrium measures. Furthermore, if there exists a $(K,t)$-equilibrium measure $\tilde{\mu}$ not contained in $\CC$, then, by the same reason, there exists an ergodic $(K,t)$-equilibrium measure $\tilde{\nu}$ which is not in $\CC$. Since any two different ergodic measures are mutually singular, for each $i$ there is a Borel set $A_i$ such that $\nu_i(A_i)=0$ and $\tilde{\nu}(A_i)=1$. Letting $A=\bigcap_{i=1}^\infty A_i$, we see that $\mu(A)=0$ and $\tilde{\nu}(A)=1$. Therefore, $\tilde{\nu}$ is not absolutely continuous with respect to $\mu$, which is a contradiction.

  The last assertion is immediate, since in this case $\CC = \{ \mu \}$. This completes the proof.
\end{proof}

\section{Topological pressure} \label{sec:topo}

We will now look at the topological pressure in more detail. The existence of an equilibrium measure allows us to study the differentiability of the pressure. If the ergodic semiconformal measure exists, then we are able to determine the derivative. Suppose that for each $i \in I$ there is an invertible matrix $A_i \in \R^{d \times d}$ with $\| A_i \| \le \ualpha$. We set $\lalpha = \min_{i \in I} \alpha_d(A_i)>0$.

\begin{lemma} \label{thm:lyapunov}
  If $K \subset I^\infty$ is a nonempty compact set and $\mu \in \MM_\sigma(K)$, then there exist $0 > \log\ualpha \ge \lambda_1(\mu) \ge \cdots \ge \lambda_d(\mu) \ge \log\lalpha  > -\infty$ such that
  \begin{equation*}
    \lim_{n \to \infty} \tfrac1n \sum_{\iii \in K_n} \mu([\iii])
    \log\alpha_l(A_\iii) = \lambda_l(\mu)
  \end{equation*}
  whenever $l \in \{ 1,\ldots,d \}$. Moreover, if $\mu \in \EE_\sigma(K)$, then
  \begin{equation*}
    \lambda_l(\mu) =
    \lim_{n \to \infty} \tfrac1n \log\alpha_l(A_{\iii|_n})
  \end{equation*}
  for all $l \in \{ 1,\ldots,d \}$ and for $\mu$-almost every $\iii \in K$.
\end{lemma}

\begin{proof}
  If $\mu \in \MM_\sigma(K)$, we define $\lambda_l(\mu) = \Lambda^l(\mu) - \Lambda^{l-1}(\mu)$ for $l \in \{ 1,\ldots,d \}$ to get the first claim. On the other hand, if $\mu \in \EE_\sigma(I^\infty)$, it follows from \eqref{eq:cylinder2} and Kingman's subadditive ergodic theorem \cite{Steele1989} that
  \begin{equation*}
    \lim_{n \to \infty} \tfrac1n \log\fii^t(A_{\iii|_n}) = \Lambda^t(\mu)
  \end{equation*}
  for $\mu$-almost every $\iii \in I^\infty$ whenever $t \ge 0$. The proof is finished.
\end{proof}

Suppose that $K \subset I^\infty$ is a nonempty set with $\sigma(K) \subset K$ and $t \in (0,\infty)$. The left and the right derivatives of the topological pressure at a point $t$ are defined by
\begin{align*}
  P_K'(t-) &= \lim_{\delta \uparrow 0} \frac{P_K(t+\delta)-P_K(t)}{\delta}, \\
  P_K'(t+) &= \lim_{\delta \downarrow 0} \frac{P_K(t+\delta)-P_K(t)}{\delta},
\end{align*}
respectively.

\begin{lemma} \label{thm:P_derivative}
  Suppose that $K \subset I^\infty$ is a nonempty compact set with $\sigma(K) \subset K$ and $t \in (0,d) \setminus \N$. If $\mu$ is a $(K,t)$-equilibrium measure, then
  \begin{equation*}
    P_K'(t-) \le \lambda_{l+1}(\mu) \le P_K'(t+),
  \end{equation*}
  where $l = \lfloor t \rfloor$.
  In fact, $P_K'(t) = \lambda_{l+1}(\mu)$ except for at most countably many points of $(0,d)$.
\end{lemma}

\begin{proof}
  Choosing $\delta \in \R$ so that $\lfloor t+\delta \rfloor = l$, we have
  \begin{align*}
    P_K(t+\delta) &\ge h(\mu) + \Lambda^{t+\delta}(\mu) = h(\mu) + \Lambda^t(\mu) + \delta\lambda_{l+1}(\mu) \\
    &= P_K(t) + \delta\lambda_{l+1}(\mu)
  \end{align*}
  by Lemma \ref{thm:lyapunov}. This gives the first claim. The second claim follows from Lemma \ref{thm:P_convex} by recalling some of the basic properties of convex functions, see, for example, \cite[Theorem 24.1]{Rockafellar1970} and \cite[Lemma 3.12 in \S 3]{SteinShakarchi2005}.
\end{proof}

\begin{remark}
  (1) If $P_K'(t)$ exists and $\mu_1,\mu_2$ are ergodic $(K,t)$-equilibrium measures, then $\lambda_{l+1}(\mu_1) = \lambda_{l+1}(\mu_2)$, where $l = \lfloor t \rfloor$.

  (2) If $K \subset I^\infty$ is a nonempty compact set with $\sigma(K) \subset K$ and $t \in [0,1)$, then there exists a $(K,t)$-equilibrium measure $\mu$ such that $P_K'(t+) = \lambda_1(\mu)$. To see this, choose a converging sequence $\{ \mu_n \}_n$, where each $\mu_n$ is a $(K,t_n)$-equilibrium measure and $t_n \downarrow t$. By the upper semicontinuity of the mapping $\mu \mapsto h(\mu) + \Lambda^t(\mu)$ defined on $\MM_\sigma(K)$, the limiting measure is a $(K,t)$-equilibrium measure. Since \eqref{eq:cylinder2} yields $\lambda_1(\mu) = \inf_{n \in \N} \tfrac{1}{n}\sum_{\iii \in K_n} \mu([\iii]) \log\alpha_{l+1}(A_\iii)$, the claim now follows by a simple calculation.

  (3) We can apply the idea used in the previous remark to show the following: If for every $t \ge 0$ there exists a constant $D \ge 1$ such that
  \begin{equation*}
    D^{-1}\fii^t(A_\iii)\fii^t(A_\jjj) \le \fii^t(A_{\iii\jjj})
  \end{equation*}
  for all $\iii, \jjj \in I^*$, then for each $t \in (0,d) \setminus \N$ we have $P'(t) = \lambda_{l+1}(\mu)$, where $l = \lfloor t \rfloor$ and $\mu$ is the ergodic $t$-semiconformal measure. Take account of Remark \ref{rem:positive_condition}. To see this, observe first that the mapping $\mu \mapsto \Lambda^t(\mu)$ defined on $\MM_\sigma(I^\infty)$ is also lower semicontinuous. Hence the mapping $\mu \mapsto \lambda_{l+1}(\mu)$ defined on $\MM_\sigma(I^\infty)$ is easily seen to be continuous. The claim now follows from the uniqueness of the equilibrium measure.
\end{remark}

In the following theorem we show that if the ergodic semiconformal measure exists, then we can determine the derivative of the topological pressure. In Example \ref{ex:nondifferentiable}, we present a nondifferentiable pressure.

\begin{theorem} \label{thm:derivative_exists_when_semiconf}
  Suppose $K \subset I^\infty$ is a nonempty compact set with $\sigma(K) \subset K$, $t \in (0,d) \setminus \N$, and $c \ge 1$. If $\mu \in \EE_\sigma(K)$ satisfies $\mu([\iii]) \ge c^{-1}e^{-|\iii|P_K(t)}\fii^t(A_\iii)$ for all $\iii \in K_*$, then
  \begin{equation*}
    P_K'(t) = \lambda_{l+1}(\mu),
  \end{equation*}
  where $l = \lfloor t \rfloor$.
\end{theorem}

\begin{proof}
  To prove the theorem, we use methods similar to those in Heurteaux \cite{Heurteaux1998}. Let us first show that $\lambda_{l+1}(\mu) \ge P_K'(t+)$. Take $\beta < P_K'(t+) < 0$, choose $\delta > \log c/\bigl( P_K'(t+)-\beta \bigr)$, and set $\eps = c^{-1}e^{P_K'(t+)\delta}-e^{\beta\delta} > 0$. Since $\mu$ is ergodic, Lemma \ref{thm:lyapunov} implies that it suffices to find a set $C \subset I^\infty$ with $\mu(C)>0$ so that
  \begin{equation*}
    \lim_{n \to \infty} \tfrac1n \log\alpha_{l+1}(A_{\iii|_n}) \ge \beta
  \end{equation*}
  for all $\iii \in C$.

  Recall that $P_K(s) \le \tfrac{1}{n} \log\sum_{\iii \in K_n} \fii^s(A_\iii)$ for all $s \ge 0$ and $n \in \N$. Choose $n_0 \in \N$ so that $\lfloor t+\delta/n \rfloor = l$ for every $n \ge n_0$. Now Lemma \ref{thm:P_convex} implies
  \begin{equation} \label{eq:kalkyyli}
    \sum_{\iii \in K_n} \fii^{t+\delta/n}(A_\iii) \ge e^{nP_K(t+\delta/n)} \ge e^{P_K'(t+)\delta + nP_K(t)}
  \end{equation}
  for all $n \ge n_0$. Letting $C_n = \{ \iii \in K_n : \alpha_{l+1}(A_{\iii|_n}) > e^{n\beta} \}$ and $C_n' = \bigcup_{\iii \in C_n} [\iii]$, it follows from \eqref{eq:kalkyyli} and \eqref{eq:cylinder1} that for every $n \ge n_0$ we have
  \begin{align*}
    e^{P_K'(t+)\delta} e^{nP_K(t)} &\le \sum_{\iii \in K_n} \fii^{t+\delta/n}(A_\iii) \le \sum_{\iii \in K_n \setminus C_n} \fii^t(A_\iii)\alpha_{l+1}(A_\iii)^{\delta/n} + \sum_{\iii \in C_n} \fii^t(A_\iii) \\
    &\le ce^{\beta\delta} e^{nP_K(t)}\bigl( 1-\mu(C_n') \bigr) + ce^{nP_K(t)} \mu(C_n').
  \end{align*}
  Hence
  \begin{equation*}
    \mu(C_n') \ge \frac{c^{-1}e^{P_K'(t+)\delta}-e^{\beta\delta}}{1-e^{\beta\delta}} \ge \eps
  \end{equation*}
  for all $n \ge n_0$, and, consequently, $\mu(C) \ge \eps$ for $C = \bigcap_{n=1}^\infty\bigcup_{k=n}^\infty C_k'$. Since for every $\iii \in C$ and for all $n \in \N$ there is $k \ge n$ so that $\tfrac1k \log\alpha_{l+1}(A_{\iii|_k}) > \beta$, we have proven the first claim.

  Let us next show that $\lambda_{l+1}(\mu) \le P_K'(t-)$. Take $\log\lalpha \le P_K'(t-) < \beta < 0$, choose $\delta > \log c/\bigl( \beta - P_K'(t-) \bigr)$, and set $\eps = (c^{-1}e^{-P_K'(t-)\delta} - e^{-\beta\delta})/(e^{-\delta\log\lalpha}-e^{-\beta\delta}) > 0$. Let $n_0 \in \N$ be such that $\lfloor t-\delta/n \rfloor = l$ for every $n \ge n_0$. As in the proof of the first claim, it follows that
  \begin{align*}
    e^{-P_K'(t-)\delta} e^{nP_K(t)} &\le \sum_{\iii \in K_n} \fii^{t-\delta/n}(A_\iii) \le \sum_{\iii \in C_n}\fii^t(A_\iii) \alpha_{l+1}(A_\iii)^{-\delta/n} + \lalpha^{-\delta} \sum_{\iii \in K_n \setminus C_n} \fii^t(A_\iii) \\
    &\le ce^{-\beta\delta}e^{nP_K(t)}\bigl( 1-\mu(K \setminus C_n') \bigr) + ce^{-\delta\log\lalpha} e^{nP_K(t)}\mu(K \setminus C_n'),
  \end{align*}
  where $C_n = \{ \iii \in K_n : \alpha_{l+1}(A_{\iii|_n}) > e^{n\beta} \}$ and $C_n' = \bigcup_{\iii \in C_n} [\iii]$. Hence
  \begin{equation*}
    \mu(K \setminus C_n') \ge \frac{c^{-1}e^{-P_K'(t-)\delta}-e^{-\beta\delta}}{e^{-\delta\log\lalpha}-e^{-\beta\delta}} \ge \eps
  \end{equation*}
  for all $n \ge n_0$. The second claim now follows as in the proof of the first claim. The proof is finished.
\end{proof}

\section{Dimension results}

Suppose $K \subset I^\infty$ is a nonempty compact set with $\sigma(K) \subset K$. If for each $i \in I$ there are an invertible matrix $A_i \in \R^{d \times d}$ with $||A_i|| \le \ualpha < 1$ and a translation vector $a_i \in \R^d$, then both the affine IFS $\{ A_i + a_i \}_{i \in I}$ and the sub-self-affine set $E_K \subset \R^d$ are called \emph{tractable} provided that $A_i\bigl( \overline{\QQ}_d \bigr) \subset \QQ_d \cup \{ 0 \}$ for all $i \in I$ and $E_K$ is not contained in any hyperplane of $\R^d$. See \cite[Appendix A]{KaenmakiShmerkin2009} and Example \ref{ex:tractable} for details. Here $\QQ_d$ is the collection of all vectors $v \in \R^d$ with strictly positive coefficients and $\overline{\QQ}_d$ its closure.

Given a nonempty compact set $K \subset I^\infty$ with $\sigma(K) \subset K$ and a tractable affine IFS, define for $r > 0$
\begin{equation*}
  Z(r) = \bigl\{ \iii \in K_* : \diam\bigl( E_\iii \bigr) \le r <
  \diam\bigl( E_{\iii^-}) \bigr) \bigr\},
\end{equation*}
and, if in addition, $x \in E_K$, set
\begin{equation*}
  Z(x,r) = \{ \iii \in Z(r) : E_\iii \cap B(x,r) \ne \emptyset \}.
\end{equation*}
We say that a tractable sub-self-affine set $E_K$ satisfies the \emph{ball condition} if there exists a constant $0<\delta<1$ such that for each $x \in E_K$ there is $r_0>0$ such that for every $0<r<r_0$ there exists a set $\{ x_\iii \in \conv\bigl(E_\iii\bigr) : \iii \in Z(x,r) \}$ such that the collection  $\{ B(x_\iii,\delta r) : \iii \in Z(x,r) \}$ is disjoint. By $\conv(A)$ we mean the convex hull of a given set $A$. Inspecting the proof of \cite[Theorem 3.5]{KaenmakiVilppolainen2008}, we see that the ball condition is equivalent to $\sup_{x \in E_K} \limsup_{r \downarrow 0} \# Z(x,r) < \infty$.

By $\HH^t$ we mean the $t$-dimensional Hausdorff measure, see \cite[\S 4.3]{Mattila1995}. The $n$-dimensional Lebesgue measure is denoted by $\LL^n$. Finally, the Hausdorff and the upper Minkowski dimensions are denoted by $\dimh$ and $\dimm$, respectively. Consult \cite[\S 4.8 and \S 5.3]{Mattila1995} and \cite[\S 10.1]{Falconer1997}.

\begin{theorem} \label{thm:tractable}
  Suppose $K \subset I^\infty$ is a nonempty compact set with $\sigma(K) \subset K$ and $E_K$ is a tractable sub-self-affine set satisfying the ball condition. If $P_K(s)=0$ for some $0<s\le 1$, then $\dimh(E_K)=\dimm(E_K)=s$.
\end{theorem}

\begin{proof}
  Choose $0<t<s$ and let $\mu$ be an ergodic $(K,t)$-equilibrium measure. Since $\lambda_1(\mu) \le \log\ualpha < 0 < P(t)$, it follows from \cite[Proposition 4.2]{Kaenmaki2004} that
  \begin{equation*}
    \lim_{n \to \infty}\frac{\log\mu([\iii|_n])}{\log\alpha_1(A_{\iii|_n})^t} = 1 - \frac{P(t)}{t\lambda_1(\mu)} > 1
  \end{equation*}
  for $\mu$-almost all $\iii \in K$. Now, using Egorov's Theorem, we find a compact set $C \subset K$ with $\mu(C) \ge \tfrac12$ and $n_0 \in \N$ so that $\mu([\iii|_n]) < \alpha_1(A_{\iii|_n})^t$ for all $\iii \in C$ and $n \ge n_0$. Hence, recalling \cite[Lemma A.3]{KaenmakiShmerkin2009}, there are constants $c',c \ge 1$ such that
  \begin{equation*}
    \mu|_C([\iii]) \le c'\alpha_1(A_\iii)^t \le c\diam(E_\iii)^t
  \end{equation*}
  for all $\iii \in K_*$. Since the ball condition implies $\sup_{x \in E_K} \limsup_{r \downarrow 0} \# Z(x,r) < \infty$, there exists $M>0$ such that for each $x \in E_K$ there is $r_0>0$ so that for every $0<r<r_0$ we have $\# Z(x,r) < M$ and
  \begin{equation*}
    \mu|_C \circ \pi_a^{-1}\bigl( B(x,r) \bigr) \le \sum_{\iii \in Z(x,r)} \mu([\iii]) \le c\sum_{\iii \in Z(x,r)} \diam(E_\iii)^t \le cMr^t.
  \end{equation*}
  It follows now from \cite[Proposition 2.2(a)]{Falconer1997} that $\HH^t(E_K)>0$. The proof is finished since $0<t<s$ was arbitrary and $s$ serves as an upper bound for the upper Minkowski dimension.
\end{proof}

The following theorem generalizes \cite[Theorem 5.3]{Falconer1988} and \cite[Theorem 4.5]{Kaenmaki2004} to sub-self-affine sets. We remark that the claim $\dimh(E_K) = \min\{ s,d \}$ in the theorem could alternatively be proved by modifying the proof of \cite[Theorem 5.3]{Falconer1988} to the setting of sub-self-affine sets. However, taking advantage of the existence of an equilibrium measure permits a stronger statement with a simple proof.

\begin{theorem} \label{thm:ssa_ae_dim}
  Suppose for each $i \in I$ there is an invertible matrix $A_i \in \R^{d \times d}$ with $\|A_i\| < \tfrac12$. If $K \subset I^\infty$ is a nonempty compact set with $\sigma(K) \subset K$ and $P_K(s)=0$, then for $\LL^{d\kappa}$-almost every choice of $a=(a_1,\ldots,a_\kappa) \in \R^{d\kappa}$, where $a_i \in \R^d$ is a translation vector and $\kappa = \# I$, we have
  \begin{equation*}
    \dimh(E_K) = \dimh(\mu \circ \pi_a^{-1}) = \min\{ s,d \},
  \end{equation*}
  where $E_K = \pi_a(K)$ and $\mu$ is an ergodic $(K,s)$-equilibrium measure.
\end{theorem}

\begin{proof}
  If $\mu \in \EE_\sigma(K)$ and $\diml(\mu) = l - \bigl( h(\mu) + \Lambda^l(\mu) \bigr)/\lambda_{l+1}(\mu)$, where $l = \max\{ k \in \N : 0 < h(\mu) + \Lambda^k(\mu) \}$, then \cite[Theorem 1.9]{JordanPollicottSimon2007} implies that $\dimh(\mu \circ \pi_a^{-1}) = \min\{ \diml(\mu),d \}$ for $\LL^{d\kappa}$-almost all $a \in \R^{d\kappa}$. A simple calculation shows that if $\mu$ is a $(K,s)$-equilibrium measure, then $\diml(\mu)=s$ if and only if $P_K(s)=0$. The existence of an ergodic $(K,s)$-equilibrium measure finishes the proof.
\end{proof}

Falconer \cite{Falconer1995} asked if $\dimh(E)=\dimm(E)$ for all sub-self-similar sets $E$. Since the singularity dimension is an upper bound for the upper Minkowski dimension, Theorem \ref{thm:ssa_ae_dim} gives a partial positive answer to this question. The question remains open for sub-self-similar sets which do not satisfy the open set condition and are constructed by using exceptional (in the sense of Theorem \ref{thm:ssa_ae_dim}) translation vectors.

\section{Remarks and examples}

In this last section, we give the examples mentioned in the previous sections. We begin by recalling a geometric condition implying the uniqueness of the equilibrium measure.

\begin{remark} \label{rem:positive_condition}
  Suppose that for each $i \in I$ there is a contractive invertible matrix $A_i \in \R^{2 \times 2}$ such that the following geometric condition is satisfied: There exist $\theta \in S^1$ and $0 < \beta < \pi/2$ such that
  \begin{align*}
    A_i\bigl( \overline{X(\theta,\beta)} \bigr) &\subset
    X(\theta,\beta), \\
    A_i^*\bigl( \overline{X(\theta,\beta)} \bigr) &\subset
    X(\theta,\beta)
  \end{align*}
  for all $i \in I$. Here $A^*$ denotes the transpose of a given matrix $A$, $\overline{B}$ the closure of a given set $B$, and $X(\theta,\beta) = \{ x \in \R^2 : \cos(\beta/2) < |\theta \cdot x| / |x|, \; x \ne 0 \} \cup \{ 0 \}$. Then it follows from \cite[Lemma 4.1]{KaenmakiShmerkin2009} that there exists a constant $D \ge 1$ such that for every $0 \le t \le 2$
  \begin{equation*}
    D^{-1}\fii^t(A_\iii)\fii^t(A_\jjj) \le \fii^t(A_{\iii\jjj})
  \end{equation*}
  whenever $\iii,\jjj \in I^*$. Consequently, there exists an ergodic $t$-semiconformal measure. Theorem \ref{thm:semiconformal_unique} implies that it is the only $t$-equilibrium measure.
\end{remark}

We calculate the topological pressure for $0 \le t \le 1$ when the $2 \times 2$ matrices are diagonal. Let $I = \{ 0,1 \}$ and
\begin{equation*}
  A_0 =
  \begin{pmatrix}
    \beta & 0 \\
    0 & \gamma
  \end{pmatrix}, \quad
  A_1 =
  \begin{pmatrix}
    \lambda & 0 \\
    0 & \theta
  \end{pmatrix},
\end{equation*}
where $0 < \beta,\gamma,\lambda,\theta < 1$. Suppose $\iii \in I^n$ has $k$ zeros and $n-k$ ones. Notice that $A_\iii$ is a diagonal matrix with diagonal elements, which are also its singular values, $\beta^k \lambda^{n-k}$ and $\gamma^k \theta^{n-k}$. Hence, if $0 \le t \le 1$, then
\begin{equation*}
  \fii^t(A_\iii) = \max\{ \beta^k\lambda^{n-k},\gamma^k\theta^{n-k} \}^t.
\end{equation*}
Therefore,
\begin{align*}
  \max\{ (\beta^t+\lambda^t&)^n,(\gamma^t+\theta^t)^n \} \le
  \sum_{k=0}^n \binom{n}{k} \max\{ \beta^k\lambda^{n-k},
  \gamma^k\theta^{n-k} \}^t \\
  &= \sum_{\iii \in I^n} \fii^t(A_\iii)
  \le \sum_{k=0}^n \binom{n}{k}
  (\beta^{tk}\lambda^{t(n-k)} + \gamma^{tk}\theta^{t(n-k)}) \\
  &= (\beta^t+\lambda^t)^n + (\gamma^t+\theta^t)^n
  \le 2\max\{ (\beta^t+\lambda^t)^n,(\gamma^t+\theta^t)^n \},
\end{align*}
and consequently,
\begin{equation} \label{eq:diag_topo}
  P(t)=\max\{ \log(\beta^t + \lambda^t), \log(\gamma^t + \theta^t) \}
\end{equation}
for $0 \le t \le 1$.

\begin{example} \label{ex:not_unique}
  We exhibit a self-affine set $E$ having exactly two different ergodic $s$-equilibrium measures and, as a consequence, infinitely many $s$-equilibrium measures, where $P(s)=0$. We also show that any convex combination of these two ergodic $s$-equilibrium measures is $s$-semiconformal and $\PP^s(E)<\infty$.

  Let $I = \{ 0,1 \}$ and
  \begin{equation*}
    A_0 =
    \begin{pmatrix}
      \lambda & 0 \\
      0 & \gamma
    \end{pmatrix}, \quad
    A_1 =
    \begin{pmatrix}
      \gamma & 0 \\
      0 & \lambda
    \end{pmatrix}
  \end{equation*}
  where $0<\gamma<\lambda\le\tfrac12$. Now the mappings of the affine IFS $\{ A_0, A_1 + (1-\gamma,1-\lambda) \}$ map the unit square as illustrated in Figure \ref{fig:not_unique}.
  \begin{figure}
    \begin{center}
      \includegraphics[scale=0.6]{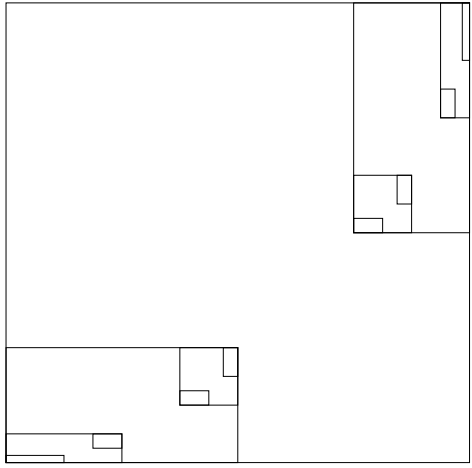}
    \end{center}
    \caption{An illustration for Example \ref{ex:not_unique} when
      $\lambda=\tfrac12$ and $\gamma=\tfrac14$.}
    \label{fig:not_unique}
  \end{figure}
  Let $\pi \colon I^\infty \to \R^2$ be the projection mapping associated to this affine IFS and $E=\pi(I^\infty)$ the corresponding self-affine set. Let us also set $E_x = \proj_x(E)$ and $E_y = \proj_y(E)$, where $\proj_x$ and $\proj_y$ are orthogonal projections onto the $x$-axis and $y$-axis, respectively. We choose $0 < s < 1$ to be the unique number for which $\lambda^s + \gamma^s = 1$. It is easy to see that $E_x$ and $E_y$ are self-similar sets and $s = \dimh(E_x) = \dimh(E_y) \le \dimh(E) \le 1$. It follows from \eqref{eq:diag_topo} that $P(t)=\log(\lambda^t+\gamma^t)$, yielding $\dimh(E) \le s$. Therefore $\dimh(E)=s$.

  Let $\mu$ and $\nu$ be the ergodic Bernoulli measures on $I^\infty$ obtained from probability vectors $(\lambda^s,\gamma^s)$ and $(\gamma^s,\lambda^s)$, respectively. Consult, for example, \cite[Theorem 2.2]{KaenmakiVilppolainen2008}. It follows from standard arguments that $\mu \circ (\proj_x \pi)^{-1}$ has full dimension on $E_x$ and $\nu \circ (\proj_x \pi)^{-1}$ has full dimension on $E_y$. Hence $\dimh(\mu \circ \pi^{-1}) = \dimh(\nu \circ \pi^{-1}) = \dimh(E) = s$. Using \cite[Theorem 1.9]{JordanPollicottSimon2007}, we get $s = -h(\mu)/\lambda_1(\mu) = -h(\nu)/\lambda_1(\nu)$, and thus $h(\mu) + \Lambda^s(\mu) = h(\mu) + s\lambda_1(\mu) = 0 = P(s) = h(\nu) + \Lambda^s(\nu)$ yielding that both the measures $\mu$ and $\nu$ are $s$-equilibrium measures. See also \cite{Rams2005}.

  The convexity of the set of all $s$-equilibrium measures implies that any convex combination of the measures $\mu$ and $\nu$ is an $s$-equilibrium measure. We claim that any such a measure $\eta$ is an $s$-semiconformal measure. Since $P(s)=0$, we have to check there exists a constant $c \ge 1$ such that
  \begin{equation*}
    c^{-1} \fii^s(A_\iii) \le \eta([\iii]) \le c \fii^s(A_\iii)
  \end{equation*}
  for all $\iii\in I^*$. But this follows immediately from the fact that
  \begin{equation*}
    \fii^s(A_{\iii}) = \max\{ \lambda^k \gamma^{n-k},
    \lambda^{n-k} \gamma^k \}^s.
  \end{equation*}
  We can now apply Theorem \ref{thm:semiconformal_unique} to conclude that the set of $s$-equilibrium measures is the convex hull of $\{\mu,\nu\}$ and, in particular, that $\mu$ and $\nu$ are the only ergodic $s$-equilibrium measures.

  That $\PP^s(E)<\infty$ follows now from Remark \ref{rem:finite_measure}.
\end{example}

\begin{question}
  Is the number of different ergodic $t$-equilibrium measures on a self-affine set always finite? In a forthcoming paper \cite{FengKaenmaki2009}, we show that in $\R^2$ a self-affine set can have at most two different ergodic $t$-equilibrium measures.
\end{question}

\begin{example} \label{ex:no_semiconformal}
  In this example, we show that semiconformal measures do not always exist. We also exhibit a self-affine set $E$ for which $\HH^s(E) = \infty$, where $s$ is the singularity dimension.

  Let $I=\{0,1\}$ and set
  \begin{equation*}
    A_0 = A_1 = \lambda
    \begin{pmatrix}
      1 & 1 \\
      0 & 1
    \end{pmatrix},
  \end{equation*}
  where $0<\lambda<1/2$ is fixed. Observe that the equilibrium measure is unique by Lemma \ref{thm:equal_maps}. Let us show by a direct calculation that there cannot be an $s$-semiconformal measure when $P(s)=0$. It is straightforward to see that if $\iii \in I^n$, then
  \begin{equation*}
    A_\iii = \lambda^n
    \begin{pmatrix}
      1 & n \\
      0 & 1
    \end{pmatrix},
  \end{equation*}
  and, consequently, $\alpha_1(A_\iii)^2 = n^2\lambda^{2n} \bigl( 1/n^2 + 1/2 + (1/n^2 + 1/4)^{1/2} \bigr)$. This, in turn, shows that
  \begin{equation*}
    (n \lambda^n)^t \le \fii^t(A_\iii) \le 2 (n \lambda^n)^t
  \end{equation*}
  whenever $0 < t < 1$. We deduce that $P(t) = \log(2\lambda^t)$ for $0<t<1$, yielding the singular value dimension to be $s=\log 2/\log(1/\lambda)$. Suppose $\mu$ is an $s$-semiconformal measure, let $c$ be as in \eqref{eq:semiconformal}, and pick an integer $n > c^{1/s}$. Then
  \begin{equation*}
    1 = \sum_{\iii\in I^n} \mu([\iii]) \ge c^{-1} \sum_{\iii\in I^n}
    \fii^s(A_\iii) \ge c^{-1} 2^n (n\lambda^n)^s = c^{-1}n^s,
  \end{equation*}
  giving the desired contradiction.

  Let us consider the affine IFS $\{ A_0,A_1+(1,1) \}$, and denote the invariant set by $E$. In what follows, we will show that $\HH^s(E)=\infty$. Let $P(x,y) = x-y$. This is a Lipschitz map, so it is enough to show that $\HH^s\bigl( P(E) \bigr) = \infty$. Note that
  \begin{equation*}
    P\bigl( \pi(\iii) \bigr) = \sum_{k=0}^\infty k\lambda^k i_{k+1}
  \end{equation*}
  for any $\iii = (i_1,i_2,\ldots) \in I^\infty$. In particular, if $\iii \in I^n$, then
  \begin{equation*}
    \diam\bigl( P(E_\iii) \bigr) = \sum_{k=n}^\infty
    k\lambda^k = \lambda^n \frac{1 + (n-1)(1-\lambda)}{(1-\lambda)^2} =:
    \delta_n.
  \end{equation*}
  A short calculation yields
  \begin{equation*}
    \delta_{n}-2\delta_{n+1} = \frac{\lambda^n}{(1-\lambda)^2}
    \bigl( n(1-3\lambda+2\lambda^2)-\lambda \bigr).
  \end{equation*}
  Notice that $1 - 3\lambda + 2\lambda^2 > 0$ since $\lambda < 1/2$. Thus we may choose $n_0$ such that $\delta_n - 2\delta_{n+1} > 0$ for all $n \ge n_0$.

  Fix any $\iii \in I^{n_0}$. The choice of $n_0$ guarantees that for any two incomparable $\jjj,\jjj' \in I^*$, the sets $P(E_{\iii\jjj})$ and $P(E_{\iii\jjj'})$ are disjoint. We can therefore construct a Borel probability measure $\nu$ supported on $F := P(E_\iii)$, by assigning
  equal mass $2^{-n}$ to all the sets $P(E_{\iii\jjj})$ where $\jjj \in I^n$. Recall that $2^{-n}=\lambda^{ns}$ and pick $x \in F$ and $0 < r < \delta_{n_0+1}$. Noting that $\{ \delta_n \}$ is decreasing in $n$, we may fix $n$ for which $\delta_n < r \le \delta_{n-1}$. We obtain
  \begin{equation*}
    \frac{\nu\bigl( B(x,r) \bigr)}{r^s} \le \frac{\nu\bigl(
      B(x,\delta_{n-1}) \bigr)}{\delta_n^s} \le
    \frac{3\lambda^{(n-1)s}}{\delta_n^s}
  \end{equation*}
  since $P(E_{\iii\jjj}) \cap B(x,\delta_{n-1}) \ne \emptyset$ for at most three different $\jjj \in I^{n-1}$. Notice that $\lim_{n \to \infty} \lambda^n/\delta_n = 0$. Now it follows from \cite[proof of Theorem 5.7]{Mattila1995} that $\HH^s(F) = \infty$. Thus $\HH^s(E) = \infty$, as claimed.
\end{example}

\begin{example} \label{ex:nondifferentiable}
  We exhibit a nondifferentiable topological pressure. This example is a modification of \cite[Example 3.5]{FengLau2002}. Let $I = \{ 0,1 \}$ and
\begin{equation*}
  A_0 =
  \begin{pmatrix}
    \tfrac14 & 0 \\
    0 & \tfrac{1}{32}
  \end{pmatrix}, \quad
  A_1 =
  \begin{pmatrix}
    \tfrac14 & 0 \\
    0 & \tfrac12
  \end{pmatrix}.
\end{equation*}
  According to \eqref{eq:diag_topo}, we now have $P(t) = \max\{ P_1(t),P_2(t) \}$, where $P_1(t) = (1-2t)\log 2$ is affine and $P_2(t) = \log(32^{-t} + 2^{-t})$ is strictly convex. It obviously follows that $P$ has a point of nondifferentiability in $(0,1)$, since $P_1(\tfrac12) > P_2(\tfrac12)$ and $P_1(1) < P_2(1)$.
\end{example}

\begin{example} \label{ex:tractable}
  In this example, we construct an affine IFS $\{ A_i + a_i \}_{i \in I}$ on $\R^2$ satisfying $A_i(\overline{\QQ}_2) \subset \QQ_2 \cup \{ 0 \}$ for all $i \in I$, but $\dimm(E) < s$ for $P(s)=0$. This shows that in the definition of the tractable affine IFS, the condition on hyperplanes is indispensable.

  Let $I = \{ 0,1 \}$ and set
  \begin{equation*}
    A_i =
    \begin{pmatrix}
      \ulambda-\beta_i & \theta_i \\
      \beta_i & \ulambda-\theta_i
    \end{pmatrix}
    =
    \begin{pmatrix}
      \theta_i+\llambda_i & \theta_i \\
      \beta_i & \beta_i+\llambda_i
    \end{pmatrix}
  \end{equation*}
  for all $i \in I$, where $\beta_1=\theta_2=17/100$, $\beta_2=\theta_1=13/100$, $\ulambda=1/3$, and $\llambda_i = \ulambda - \beta_i - \theta_i = 1/30$ for all $i \in I$. Now trivially $A_i(\overline{\QQ}_2) \subset \QQ_2 \cup \{ 0 \}$ for all $i \in I$.

  Define $L_q = \{ (x,y) \in \R^2 : x+y=q \}$ for all $q \in \R$. Observe that $\dist(L_q,L_p) = |q-p|/\sqrt{2}$. Now choosing $a_i = (1-\ulambda)(\theta_i,\beta_i)/(\beta_i+\theta_i)$ and letting $f_i = A_i + a_i$, an elementary calculation gives $f_i(L_q) = L_{1-\ulambda(1-q)}$ for all $i \in I$. In particular, $f_i(L_1)=L_1$ and $f_i(L_{1-\ulambda^n})=L_{1-\ulambda^{n+1}}$ for all $n \in \N$. Since $f_i(0,0) \in L_{1-\ulambda}$, we get $f_\iii(0,0) \in L_{1-\ulambda^n}$ for all $\iii \in I^n$. But since $f_\iii(1,0) \in L_1$ for all $\iii \in I^*$, we have
  \begin{equation*}
    \ulambda^n/\sqrt{2} = \dist(L_1,L_{1-\ulambda^n}) \le |f_\iii(1,0)-f_\iii(0,0)| \le \alpha_1(A_\iii)
  \end{equation*}
  for all $\iii \in I^n$.

  On the other hand, since $f_i(x,1-x) = \bigl( \llambda_i x, \llambda_i (1-x) \bigr) + (1-\llambda_i)(\theta_i,\beta_i)/(\beta_i+\theta_i)$ for all $i \in I$ and $x \in \R$, it follows that $f_i|_{L_1}$ is a similitude mapping acting on $L_1$ with a contraction ratio $\llambda_i$. Since both systems, $\{ f_i \}_{i \in I}$ and $\{ f_i|_{L_1} \}_{i \in I}$ have the same invariant set $E$, it follows from the choices of $\ulambda$ and $\llambda_i$ that $\dimm(E) < s$ for $P(s)=0$.

  We remark that this example also shows that the assumption \cite[Hypothesis 3]{HueterLalley1995} is indispensable in the setting of Hueter and Lalley \cite{HueterLalley1995}.
\end{example}

\begin{ack}
  AK is grateful to Thomas Jordan and Pablo Shmerkin for making helpful comments during the preparation of this article. He also thanks the University of North Texas, where this research was started, for warm hospitality.
\end{ack}


\end{document}